\documentclass[12pt,reqno]{amsart}

\usepackage{amsmath}
\usepackage{amssymb}
\usepackage{amsthm}

\newtheorem{theorem}{Theorem}[section]

\newtheorem{lemma}[theorem]{Lemma}

\theoremstyle{definition}

\newtheorem{remark}[theorem]{Remark}

\newcommand{\C}{\mathbb{C}}
\newcommand{\R}{\mathbb{R}}
\newcommand{\M}{\mathcal{M}}
\newcommand{\X}{{X_{\C}}}
\newcommand{\Sing}{\mathrm{Sing}}
\newcommand{\cL}{\mathcal{L}}

\numberwithin{equation}{section}

\begin{document}
\baselineskip=15.5pt

\title[Moduli spaces of principal bundles over $\mathbb R$]{A Torelli theorem
for moduli spaces of principal bundles on curves defined over $\mathbb R$}

\author[I. Biswas]{Indranil Biswas}

\address{School of Mathematics, Tata Institute of Fundamental
Research, Homi Bhabha Road, Mumbai 400005, India}

\email{indranil@math.tifr.res.in}

\author[O. Serman]{Olivier Serman}

\address{Univ. Lille, CNRS, UMR 8524 - Laboratoire Paul Painlev\'e, F-59000 Lille, France}

\email{Olivier.Serman@math.univ-lille1.fr}

\subjclass[2010]{14D20, 14P99, 14C34}

\keywords{Curve over $\mathbb R$, principal bundle, moduli space, semistability,
Torelli theorem}

\date{}

\begin{abstract}
Let $X$ be a geometrically irreducible smooth projective curve, of genus
at least three, defined over the field of real numbers. Let $G$ be a connected reductive
affine algebraic group, defined over $\mathbb R$, such that $G$ is
nonabelian and has one simple factor. We prove that the isomorphism
class of the moduli space of principal $G$--bundles on $X$ determine uniquely
the isomorphism class of $X$.
\end{abstract}

\maketitle

\section{Introduction}

Let $X$ be a geometrically irreducible smooth projective curve defined over
the field of real numbers, of genus $g$, with $g\,\geq\, 3$. Let $\cL\,\in\,
\text{Pic}^d(X)$
be a point defined over $\mathbb R$. We note that $\cL$ need not correspond to a
line bundle over $X$. For example, the unique $\mathbb R$--point of $\text{Pic}^1$
of the anisotropic conic does not correspond to a line bundle over the anisotropic
conic. Let
${\mathcal N}_X(r,\cL)$ denote the moduli space of semistable vector bundles on $X$
of rank $r$ and determinant $\cL$, where $r\, \geq\, 2$.

We prove that the isomorphism class of the variety ${\mathcal N}_X(r,\cL)$
uniquely determines the isomorphism class of the real curve $X$ (Theorem \ref{thm1}).

When the base field is complex numbers, this was proved in \cite{MN}
for rank two, and in \cite{Tj}, \cite[{p. 229, Theorem E}]{KP}
for general $r$ and $d$.

Let ${\X}$ be the complexification of $X$.
Let $G_\C$ be a connected reductive affine algebraic group defined over $\mathbb C$,
and let $G$ be a real form of $G_\C$. We assume that $G_\C$ is nonabelian and it has
exactly one simple factor. The antiholomorphic involution of $G_\C$
corresponding to $G$ will be denoted by $\sigma_G$.
Let $\M_{\X}(G_\C)$ denote the moduli space of topologically trivial
semistable principal $G_\C$--bundles on $\X$. The variety $\M_{\X}(G_\C)$ 
is the complexification of the component of the
moduli space of principal $G$--bundles on $X$ that contains the trivial $G$--bundle.
The involution $\sigma_G$ and the antiholomorphic
involution of $\X$ together produce the antiholomorphic involution
$\sigma_{\mathcal M}$ of $\M_{\X}(G_\C)$.

We prove that the isomorphism class of the real variety $(\M_{\X}(G_\C)\, ,
\sigma_{\mathcal M})$ uniquely determines the isomorphism class of $X$ (Theorem
\ref{thm2}).

The proof of Theorem \ref{thm2} crucially uses a result of \cite{BH} which says that the
isomorphism class of $\M_{\X}(G_\C)$ uniquely determines the isomorphism
class of $\X$.

\section{Moduli spaces of vector bundles}

Let $X$ be a geometrically irreducible smooth projective curve defined over
$\mathbb R$. Let $g$ denote the genus of $X$. We will assume that $g\,\geq\, 3$.
For any $d\,\in\, \mathbb Z$ and any integer $r\, \geq\, 2$, let ${\M}_X(r,d)$ be the
moduli space of semistable vector bundles on $X$ of rank $r$ and degree $d$;
see \cite{BHu}, \cite{BGH}, \cite{BHH}, \cite{Sc1}, \cite{Sc2}, \cite{Sc3} for
moduli spaces of bundles over $X$. Let
$$
\det\, :\, {\M}_X(r,d)\,\longrightarrow\, \text{Pic}^d(X)
$$
be the morphism defined by $E\, \longmapsto\, \bigwedge^r E$. Take any
$\mathbb R$--point $\cL\,\in\, \text{Pic}^d(X)$. Define
$$
{\mathcal N}_X(r,\cL)\, :=\, {\det}^{-1}(\cL)\,\subset\, {\M}_X(r,d)\, .
$$
This ${\mathcal N}_X(r,\cL)$ is a geometrically irreducible normal projective variety
defined over $\mathbb R$, of dimension $(r^2-1)(g-1)$.

Let $\X\, :=\, X_{\C}\,=\, X\times_{\mathbb R}\C$ be the
complex projective curve obtained from $X$ by extending the base field to $\C$.
Let $\cL_\C \,\in\, \text{Pic}^d(\X)$ be the pull-back of $\cL$ to
$\X$ by the natural morphism
$\xi\, :\, \X\,\longrightarrow\, X$. The nontrivial element of the
Galois group $\text{Gal}(\xi)\,=\,\text{Gal}({\mathbb C}/{\mathbb R})\,=\,
{\mathbb Z}/2\mathbb Z$ produces an antiholomorphic involution
$$
\sigma\, :\, \X \,\longrightarrow \,\X\, .
$$

The conjugate vector bundle of a holomorphic vector bundle $E$ on $\X$
will be denoted by $\overline{E}$. We recall that the underlying
real vector bundle for $\overline{E}$ is identified with that of $E$, while the
multiplication on $\overline{E}$ by any $c\, \in\, \C$ coincides with the
multiplication by $\overline c$ on $E$. The $C^\infty$ vector bundle
$\sigma^\ast\overline E$ has a natural holomorphic structure which is uniquely
determined by the condition that the natural $\mathbb R$--linear identification of
it with $E$ is anti-holomorphic. Note that we have a commutative diagram
$$
\begin{matrix}
E & \stackrel{\sim}{\longrightarrow} & \sigma^\ast\overline E\\
\Big\downarrow && \Big\downarrow\\
X & \stackrel{\sigma}{\longrightarrow} & X
\end{matrix}
$$
It is easy to see that $E$ is semistable (respectively, stable) if and only if 
$\sigma^\ast\overline E$ is semistable (respectively, stable). Similarly, $E$ is 
polystable if and only if $\sigma^\ast\overline E$ is polystable.

The above hypothesis that $\cL \,\in\, \text{Pic}^d(X)$ means that the line
bundle $\cL_\mathbb C$ is holomorphically isomorphic to
the line bundle $\sigma^\ast\overline {\cL_{\mathbb{C}}}$.

Let ${\mathcal N}_{\X}(r,\cL_\C)$ be the moduli space of semistable vector
bundles on $\X$ of rank $r$ and determinant $\cL_\C$. The complex variety
${\mathcal N}_{\X}(r,\cL_\C)$ coincides with the complexification
${{\mathcal N}_{X}(r,\cL) \times_\mathbb R \C}$ of $\mathcal N_{X}(r,\cL)$; the
resulting antiholomorphic involution
\begin{equation}\label{e1}
\sigma_{\mathcal N}\, :\, {\mathcal N}_{\X}(r,\cL_\C)\,\longrightarrow\,
{\mathcal N}_{\X}(r,\cL_\C)
\end{equation}
sends a vector bundle $E$ on $\X$ to the vector bundle $\sigma^\ast
\overline{E}$.

\begin{theorem}\label{thm1}
The isomorphism class of the $\mathbb R$--variety $\mathcal {N}_{X}(r,\cL)$
uniquely determines the isomorphism class of the real curve $X$.
\end{theorem}

\begin{proof}
First note that the isomorphism class of the complex variety
$\mathcal{N}_{\X}(r,\cL_\C)$ uniquely determines the complex curve
$\X$ \cite{Tj}, \cite[p. 229, Theorem E]{KP}. We have to 
prove that the antiholomorphic involution $\sigma_{\mathcal N}$ determines $\sigma$. 

Let $\tau$ be an antiholomorphic involution of $\X$ such that the involution 
$E\, \longmapsto\, \tau^\ast\overline{E}$ preserves $\mathcal N_{\X} (r, \cL_\C)$.
The resulting antiholomorphic involution of ${\mathcal N}_{\X}(r,\cL_\C)$
will be denoted by $\tau_{\mathcal N}$. The two real varieties
$(\mathcal N_{\X} (r,\cL_\C)\, , \tau_{\mathcal N})$
and $(\mathcal N_{\X}(r,\cL_\C)\, ,
\sigma_{\mathcal N})$ are isomorphic if and only if there exists a complex algebraic
automorphism $f$ of ${\mathcal N}_{\X}(r,\cL_\C)$ such that 
\begin{equation}\label{f1}
\tau_{\mathcal N} \,=\, f^{-1} \sigma_{\mathcal N} f\, .
\end{equation}
Assume that the two real varieties $(\mathcal N_{\X}
(r,\cL_\C)\, , \tau_{\mathcal N})$ and $(\mathcal N_{\X}(r,\cL_\C)\, ,
\sigma_{\mathcal N})$ are isomorphic. Fix an automorphism $f$ of
$\mathcal N_{\X}(r,\cL_\C)$ satisfying \eqref{f1}.

The dual a vector bundle $E$ will be denoted by $ E^\vee$; the dual of a
line bundle $\nu$ will also be denoted by $\nu^{-1}$.

Take any algebraic automorphism $h$ of ${\mathcal N}_{\X}
(r,L_\C)$. It follows from \cite[p. 228, Theorem B]{KP} and
\cite[p. 228, remark 0.1]{KP} that $h$ is either of the form
$E \,\longmapsto\, H^* E\otimes \nu$ or 
$E \,\longmapsto \, H^* E^\vee \otimes \nu_1$, where $H$ is an
automorphism of $\X$ uniquely determined by $h$
while $\nu$ a line bundle on $\X$
with $\nu^{\otimes r}\, =\, {\mathcal O}_X$ and $\nu_1$ a line bundle on $\X$ 
with $\nu^{\otimes r}_1\, =\, L^{\otimes 2}_\C$; it should be
clarified both $\nu$ and $\nu_1$ are independent of $E$. Therefore, we get a map
\begin{equation}\label{Psi}
\Psi\, :\, \text{Aut}({\mathcal N}_{\X}
(r,L_\C))\, \longrightarrow\, \text{Aut}(\X)\, ,\ \
h\, \longmapsto\, H^{-1}\, .
\end{equation}
It is straight-forward to check that $\Psi$ is a homomorphism of groups.

We will denote $\Psi(f)\, \in\, \text{Aut}(\X)$ by $\varphi$, where $\Psi$
is defined in \eqref{Psi} and $f$ is the automorphism in \eqref{f1}. First assume that
$$
f(V)\, =\, A\otimes \varphi^\ast V\, ,
$$
where $A$ is a line bundle on $\X$. Therefore, we have
$$
f^{-1}(V)\, =\, ((\varphi^{-1})^\ast A^{-1})\otimes (\varphi^{-1})^\ast V\, .
$$
Hence the automorphism $\tau^{-1}_{\mathcal N}\circ f^{-1}\circ
\sigma_{\mathcal N}\circ f$ of ${\mathcal N}_{\X}(r, L_\C)$ is the morphism defined by
$$
V\, \longmapsto\, \tau^\ast \overline{((\varphi^{-1})^{^\ast} A^{-1})\otimes
(\varphi^{-1})^{^\ast} ((\sigma^*{\overline A})\otimes (\sigma^*
\overline{\varphi^*V}))}
$$
$$
=\, B\otimes \tau^\ast(\varphi^{-1})^\ast\sigma^\ast \varphi^*V\,=\,B\otimes
(\varphi\circ\sigma\circ \varphi^{-1}\circ\tau)^\ast V\, ,
$$
where $B$ is a line bundle which does not depend on $V$. This implies that
\begin{equation}\label{j1}
\eta\,:=\, \Psi(\tau^{-1}_{\mathcal N}\circ f^{-1}\circ\sigma_{\mathcal N}\circ f)
\,=\, \varphi\circ\sigma\circ \varphi^{-1}\circ\tau\, .
\end{equation}
Now from \eqref{f1} we conclude that $\eta\,=\, \text{Id}_{\X}$. So from
\eqref{j1} we have
$$
\tau\,=\, \varphi\circ \sigma\circ\varphi^{-1}\, .
$$
Therefore, $\varphi$ produces an isomorphism between the two curves
$({\X}\, ,\sigma)$ and $({\X}\, ,\tau)$.

Next assume that
$$
f(V)\, =\, A\otimes \varphi^\ast V^\vee\, ,
$$
where $A$ is a line bundle on $\X$. Then
$$
f^{-1}(V)\, =\, ((\varphi^{-1})^\ast A)\otimes (\varphi^{-1})^\ast V^\vee\, .
$$
Therefore, the automorphism $\tau^{-1}_{\mathcal N}\circ f^{-1}\circ
\sigma_{\mathcal N}\circ f$ of ${\mathcal N}_{\X}(r, L_\C)$ is the morphism defined by
$$
V\, \longmapsto\, \tau^\ast \overline{((\varphi^{-1})^{^\ast} A)\otimes
((\varphi^{-1})^{^\ast} ((\sigma^*{\overline A})\otimes (\sigma^*
\overline{\varphi^*V^\vee}))^\vee)}
$$
$$
=\, B\otimes \tau^\ast(\varphi^{-1})^\ast\sigma^\ast \varphi^*V\,=\,B\otimes
(\varphi\circ\sigma\circ \varphi^{-1}\circ\tau)^\ast V\, ,
$$
where $B$ is a line bundle which does not depend on $V$. This implies that
$$
\Psi(\tau^{-1}_{\mathcal N}\circ f^{-1}\circ\sigma_{\mathcal N}\circ f)
\,=\, \varphi\circ\sigma\circ \varphi^{-1}\circ\tau\, .
$$
Hence, as before, $\tau\,=\, \varphi\circ \sigma\circ\varphi^{-1}$.
This completes the proof of the theorem.
\end{proof}

\section{Moduli spaces of principal bundles}

Let $G_\C$ be a connected nonabelian reductive group over $\C$ with only one 
simple factor and let 
$$
\sigma_G \,\colon\, G_\C \,\longrightarrow\, G_\C
$$
be an antiholomorphic automorphism of order two. We denote by $G$ the real form of
$G_\C$ corresponding to $\sigma_G$.

Let $\M_{\X}(G_\C)$ denote the moduli space of topologically trivial
semistable principal $G_\C$--bundles on $\X$. It is an irreducible
normal projective variety defined over $\mathbb C$. For any holomorphic
principal $G_\C$--bundle $E$ on $\X$, let
$$
{\overline E}\,=\, E(\sigma_G)\,=\,E \times^{\sigma_{_G}} G_\C\,\longrightarrow\,
\X
$$
be the $C^\infty$ principal $G_\C$--bundle
obtained by twisting the action of $G_\C$ using the homomorphism $\sigma_G$. So the
total space of $\overline E$ is identified with that of $E$, but the action of
any $y\, \in\, G_\C$ on $\overline E$ is the action of $\sigma_G(y)$ on $E$ in
terms of the identification
of $E$ with ${\overline E}$. The pullback $\sigma^\ast \overline E$ has a holomorphic
structure uniquely determined by the condition that the above identification
between the total
spaces of $E$ and $\sigma^\ast \overline E$ is anti-holomorphic; since the total spaces
of $\overline{E}$ and $\sigma^\ast \overline E$ are naturally identified, the above
identification between the total spaces of $E$ and $\overline E$ produces an
identification of the total spaces of $E$ and $\sigma^\ast \overline E$. 
The complex projective variety $\M_{\X}(G_\C)$ carries a real structure
associated to the antiholomorphic involution
$$\sigma_{\M} \,\colon\, \M_{\X}(G_\C)\,\longrightarrow\, \M_{\X}(G_\C)\, ,
\ \ E \,\longmapsto \,\sigma^\ast \overline E\, .$$

Let $\M_X(G)$ denote the variety over $\mathbb R$ defined by
the above pair $(\M_{G_\C}(\X)\, ,\sigma_\M)$.

A Zariski closed connected subgroup $P\, \subset\, G_\C$ is called a parabolic
subgroup if $G_C /P$ is a complete variety. A Levi subgroup of $P$ is a maximal
connected reductive
subgroup of $P$ containing a maximal torus. Any two Levi subgroups of $P$ are conjugate
by some element of $P$. A proper parabolic subgroup $P\,\subset\, G_\C$
is called maximal if there is no proper parabolic subgroup of $G_\C$ containing $P$.

\begin{lemma}\label{fixedLevi}
There exists a maximal parabolic subgroup $P\, \subset\, G_\C$ and a Levi subgroup
$L\, \subset\, P$, such that the two subgroups $\sigma_G(L)$ and $L$ are conjugate by
some element of $G_\C$.
\end{lemma}

\begin{proof}
For any parabolic subgroup $P\, \subset\, G_\C$, the image $\sigma_G(P)$ is also
a parabolic subgroup of $G_\C$. Since $\sigma_G(y^{-1}Py)\,=\, \sigma_G(y)^{-1}
\sigma_G(P)\sigma_G(y)$, we get a self-map of the conjugacy classes of parabolic
subgroups of $G_\C$ that sends the conjugacy class of any $P$ to the conjugacy
class of $\sigma_G(P)$. Therefore, the involution 
$\sigma_G$ also acts on the Dynkin diagram $D$ of $G_\C$ as an involution $\tau$.
Examining the Dynkin diagrams we observe that an involution of the Dynkin diagram
of $G_\C$ must have a fixed point unless $G_\C$ is of type $A_n$ for $n$ even.

If $G_\C$ is not of type $A_n$, let $P$ be a maximal parabolic subgroup corresponding to 
a vertex of $D$ fixed by the above constructed involution $\tau$. Then $P$ and 
$\sigma_G(P)$ are conjugate in $G_\C$. Let $y\, \in\, G_\C$ be such that
$\sigma_G(P)\,=\, y^{-1}Py$. Then for any Levi subgroup $L$ of $P$,
$$
y^{-1}Ly \,\subset\, y^{-1}Py \,=\, \sigma_G(P)
$$
is a Levi subgroup of $\sigma_G(P)$.

If $G_\C$ is of type $A_n$, then $\sigma_G(L)$ and $L$ are conjugate for every Levi 
subgroup of every maximal parabolic subgroup of $G_\C$. It is enough to check this 
for $G_\C\,=\,\mathrm{SL}(n+1,\C)$, in which case this is obvious.
\end{proof}

\begin{remark}
We can be more precise as follows. The two subgroups
$\sigma_G(L)$ and $L$ are conjugate for every Levi subgroup of 
every maximal parabolic subgroup of $G_\C$ unless $G_\C$ is of type $D_n$ (with $n\,\geq
\,4$) or $E_6$.
\end{remark}

\begin{lemma}\label{pi1}
Let $L$ be any Levi subgroup of a parabolic subgroup $P$ of $G_\C$, and let
$$L'\,=\,[L,\, L]$$
be its derived subgroup. Then the homomorphism
\begin{equation}\label{hi}
\pi_1(L') \,\longrightarrow\,\pi_1(G_\C)
\end{equation}
induced by the inclusion $L'\, \hookrightarrow\, G_\C$ is injective.
\end{lemma}

\begin{proof}
Consider the fibration $L'\,\longrightarrow\, G_\C \,\longrightarrow\,
G_\C/L'$. Let
\begin{equation}\label{hi2}
\pi_2(G_\C)\,\longrightarrow\, \pi_2(G_\C/L')\,\longrightarrow\, \pi_1(L')
\,\longrightarrow\,\pi_1(G_\C)
\end{equation}
be the long exact sequence of homotopy groups associated to it. From \eqref{hi2}
we conclude that the homomorphism in \eqref{hi} is injective if
\begin{equation}\label{sh}
\pi_2(G_\C / L')\,=\, 0\, .
\end{equation}
Since $\pi_2(G_\C)\,=\, 0$ and $\pi_1(L')$ is a finite group (recall that $L'$ is
semisimple), from \eqref{hi2} it follows that $\pi_2(G_\C / L')$ is a finite group.

Now consider the fibration $P/L'\,\longrightarrow\, G_\C/L' \,\longrightarrow\, G_\C/P$.
Let
\begin{equation}\label{hi3}
\pi_2(P/L')\,\longrightarrow\, \pi_2(G_\C/L')\,\longrightarrow\, \pi_2(G_\C/P)
\,\longrightarrow\,\pi_1(G_\C)
\end{equation}
be the long exact sequence of homotopy groups 
associated to it.
Since $G_\C/P$ is simply connected, the second homotopy group $\pi_2(G_\C /P)$ is
isomorphic to $H_2(G_\C / P,\, \mathbb Z)$, which is a free abelian group. Therefore,
there is no nonzero homomorphism from the finite group $\pi_2(G_\C / L')$ to
$\pi_2(G_\C /P)$. Hence, the homomorphism
\begin{equation}\label{hi4}
\pi_2(P/L')\,\longrightarrow\, \pi_2(G_\C/L')
\end{equation}
in \eqref{hi3} is surjective.

Finally, consider the long exact sequence of homotopy groups
\begin{equation}\label{hi5}
\pi_2(L/L')\, \longrightarrow\, \pi_2(P/L')\, \longrightarrow\, \pi_2(P/L)
\end{equation}
associated to the fibration
$$
L/L' \,\longrightarrow\, P/L' \,\longrightarrow\, P/L\, .
$$
Since $P/L$ is diffeomorphic to the unipotent radical of $P$, which is contractible,
we have $\pi_2(P/L) \,=\, 0$. Also, $\pi_2(L/L') \,=\, 0$ because $L/L'$ is a Lie group.
Hence from \eqref{hi5} it follows that $\pi_2(P/L')\,=\, 0$. This implies that
\eqref{sh} holds because the homomorphism in \eqref{hi4} is surjective.
\end{proof}

\begin{theorem}\label{thm2}
The real variety $\M_X(G)$ uniquely determines the real curve $X$.
\end{theorem}

\begin{proof}
We already know that the complex variety $\M_{\X}(G_\C)\,=\,\M_X(G) \times_\R \C$
determines the complex curve $\X$ \cite{BH}.

Let $\Sing (\M_\X(G_\C))$ denote the singular locus of the variety $\M_\X(G_\C)$.
Recall from \cite{BH} that the {\it strictly semi--stable locus} $$\Delta_G\,\subset\,
\M_\X(G_\C)$$ is the Zariski closure of the set of closed points $[E] \,\in\,
\Sing (\M_\X(G_\C))$ with the property that every Euclidean neighborhood $U$ of $[E]$
contains an open neighborhood $U' \,\ni\, [E]$ such that
$U' \setminus (U' \bigcap \Sing(\M_\X(G_\C)))$ is connected and simply connected.
Moreover, this closed subset $\Delta_G$ is the union of irreducible components
corresponding to the conjugacy classes of Levi subgroups of maximal
parabolic subgroups of $G_\C$. More precisely, the decomposition of $\Delta_G$
into irreducible components is the union
$$\Delta_G\,=\,\bigcup_L M_L$$
where $L$ ranges over conjugacy classes of Levi subgroups of maximal parabolic
subgroups of $G_\C$, and $M_L$ is the image of the morphism
$\M_\X (L)\,\longrightarrow\, \M_\X(G_\C)$ given by the inclusion of
$L$ in $G_\C$. This can be deduced from
\cite[Proposition 3.1]{BH} and Lemma \ref{pi1}. Indeed, every closed point in
$\Delta_G$ is defined by a principal $G_\C$--bundle $E$ admitting a reduction of
structure group $E_L$ to a Levi subgroup $L$ of a maximal parabolic subgroup of $G_\C$.
Moreover, this $L$--bundle $E_L$ is semistable and its topological type
$\delta \,\in\, \pi_1(L)$ is torsion, which means that $\delta$ belongs to
$\pi_1([L,L])$; this is because $\pi_1(L/[L,L])$ is free abelian. Now, since $E$ is
topologically trivial, $\delta$ must be trivial (follows from
Lemma \ref{pi1}), i.e., $[E] \,\in\, \M_{\X}(G_\C)$ belongs to $M_L$.

Moreover, $M_L$ is never empty, since there always exist semi-stable principal 
$L$--bundles which are topologically trivial, for example the trivial holomorphic
principal $L$--bundle. The fact that the subvarieties $M_{L_1}$ and $M_{L_2}$
of $\M_\X(G_\C)$
are distinct when $L_1$ and $L_2$ are not conjugate by some element of
$G_\C$ is contained in the last part of the proof of \cite[Proposition 3.1]{BH}.

The antiholomorphic involution $\sigma_\M$ maps the strictly semi--stable
locus $\Delta_G$ into itself, permuting its 
irreducible components. It follows from Lemma \ref{fixedLevi} that there exists at 
least one component $M_L$ which is fixed by $\sigma_\M$. The restriction of 
$\sigma_\M$ to this component $M_L$ has at least one fixed point, namely the closed point 
corresponding to the trivial bundle.

We now proceed as in the proof of \cite[Theorem 4.1]{BH} to recover the involution 
$\sigma$ defining the real curve $X$.

First, one can assume $G_\C$ to be semi--simple. To see this, let
$Z^0_{G_\C}$ be the connected component of the center of $G_\C$
containing the identity element. Let us denote by $G'$ the
quotient $G_\C/Z^0_{G_\C}$ of $G_\C$. Note that $\sigma_G$ preserves
$Z^0_{G_\C}$, so it produces a real structure on the quotient $G'$.
The canonical line bundle of
$\M_\X(G_\C)$ (respectively, $\M_\X(G')$) will be denoted by $\omega_{\M_\X(G_\C)}$
(respectively, $\omega_{\M_\X(G')}$). We note that $\omega_{\M_\X(G')}$ pulls back to
$\omega_{\M_\X(G_\C)}$ under the morphism $\M_\X(G_\C)\,\longrightarrow\,
\M_\X(G')$ given by the quotient map $G_\C\,\longrightarrow\, G'$.
There exists an integer $m$ such that the pluri--anti--canonical system
$|-m \omega_{\M_\X(G_\C)}|$ factors into the natural map $\M_\X(G_\C)\,
\longrightarrow\, \M_\X(G')$ followed by the embedding
$$
\M_\X(G') \,\hookrightarrow\, \vert -m \omega_{\M_\X(G_\C)}\vert^\ast
\,=\,\vert -m \omega_{\M_\X(G')}\vert^\ast\, .
$$
Since the dualizing sheaves are defined
over the reals, we have real structures on $|-m \omega_{\M_\X(G_\C)}|^\ast$
and $\vert -m \omega_{\M_\X(G')}\vert^\ast$. All the maps above are defined over
$\mathbb R$. Therefore, it is enough to prove the theorem for $G'$.

So let us assume that $G_\C$ is semi--simple. We have seen above that $\Delta_G 
\,\subset\,\M_\X(G_\C)$ contains at least one irreducible component fixed by $\sigma_\M$, 
which is equal to the variety $M_L$ associated to a Levi subgroup $L$ of a maximal 
parabolic subgroup of $G_\C$. Let $$\alpha \,\colon\, \widetilde{M}_L
\,\longrightarrow\, M_L$$ be the normalization
of $M_L$, and let $\sigma_L$ be the antiholomorphic involution of $M_L$. Since
normalization commutes with the base change of field of definition, the
variety $\widetilde{M}_L$ is also defined over $\mathbb R$, and the morphism $\alpha$
is also defined over $\mathbb R$. Hence the
antiholomorphic involution $\sigma_L$ of $M_L$ lifts to $\widetilde{M}_L$. Moreover, 
$\widetilde{M}_L$ is isomorphic to the quotient $\M_\X(L)/\Gamma_L$, where $\Gamma_L$ 
is the image of $N_{G_\C}(L)$ in $\mathrm{Out}(L)$, which is either trivial or
${\mathbb Z}/2{\mathbb Z}$ (see \cite{BH}), and this quotient map is compatible with
the real structures on $\M_\X(L)$ and $\widetilde{M}_L$.

Let $Z^0_L$ be the connected component of the center of $L$
containing the identity element. Let us denote by $L'$ the quotient $L/Z_L^0$. Then
the above group $\Gamma_L$ also acts on $\M_\X(L')$, and the morphism (defined over the 
real numbers)
\begin{equation}\label{mapLevi}
\theta\,\colon\, \widetilde{M}_L\,\simeq\,\M_\X(L)/\Gamma_L\,\longrightarrow\,
\M_\X(L')/\Gamma_L
\end{equation}
can be recovered from the second tensor power of the canonical line bundle on the smooth 
locus of $\widetilde{M}_L$. Indeed, this second tensor power extends to a line bundle on 
the whole variety, and a sufficiently negative power of it gives the morphism $\theta$ 
(see \cite{BH}).

Let $\beta\, :\, \M_\X(L')\, \longrightarrow\, \M_\X(L')/\Gamma_L$ be the quotient
map. Consider $\theta$ in \eqref{mapLevi}. For any point $y\, \in\, \M_\X(L')/\Gamma_L$,
the fiber $\theta^{-1}(y)$ is $J_\X$ (respectively, $J_\X/({\mathbb Z}/2{\mathbb Z})
$) if $\# \beta^{-1}(y)\,=\, 2$ (respectively, $\# \beta^{-1}(y)\,=\, 1$).

Now take any smooth point $$y\, \in\, \M_\X(L')/\Gamma_L$$ fixed by the antiholomorphic
involution. As noted above, the fiber $\theta^{-1}(y)$ is isomorphic to either 
$J_\X$ or the singular Kummer variety $J_\X /
({\mathbb Z}/2{\mathbb Z})$. The real structure on $\theta^{-1}(y)$ induced by 
that of $\widetilde{M}_L$ comes from the real structure on the 
Jacobian associated to the curve. So in both cases we recover the Jacobian variety 
together with its natural real structure: when $\theta^{-1}(y)$ is isomorphic to $J_\X / 
({\mathbb Z}/2{\mathbb Z})$, then $J_X$ is obtained from the two--sheeted cover of
the smooth locus of the Kummer variety defined by the unique maximal torsion--free
subgroup in its fundamental group. The antiholomorphic involution can be lifted to
this cover, and this lift extends to $J_\X$ because its construction is over $\mathbb R$.

Finally, the class of the canonical principal polarization on $J_\X$ is determined as 
in \cite{BH}. Now the theorem follows from the real analog of 
Torelli theorem \cite[Theorem 9.4]{GH}.
\end{proof}

\section*{Acknowledgements}

We thank the referee for helpful comments. The first author acknowledges support
of a J. C. Bose Fellowship.



\begin{thebibliography}{AAAA}

\bibitem[BHo]{BH} I. Biswas and N. Hoffmann, A Torelli theorem for moduli spaces of 
principal bundles over a curve, {\it Ann. Inst. Fourier} {\bf 62} (2012), 
87--106.

\bibitem[BHu]{BHu} I. Biswas and J. Hurtubise, Principal bundles over a 
real algebraic curve, {\it Comm. Anal. Geom.} {\bf 20} (2012), 957--988.

\bibitem[BGH]{BGH} I. Biswas, O. Garc\'{\i}a-Prada and J. Hurtubise,
Pseudo-real principal Higgs bundles on compact Kähler manifolds, {\it Ann. Inst.
Fourier} {\bf 64} (2014), 2527--2562.

\bibitem[BHH]{BHH} I. Biswas, J. Huisman and J. Hurtubise,
The moduli space of stable vector bundles over a real algebraic
curve, \textit{Math. Ann.} \textbf{347} (2010), 201--233.

\bibitem[GH]{GH} B. H. Gross and J. Harris, Real algebraic curves, {\it Ann. Sci. 
\'Ecole Norm. Sup.} {\bf 14} (1981), no. 2, 157--182.

\bibitem[KP]{KP} A. Kouvidakis and T. Pantev, The automorphism group of the moduli
space of semistable vector bundles, {\it Math. Ann.} {\bf 302} (1995), 225--268.

\bibitem[MN]{MN} D. Mumford and P. E. Newstead, Periods of a moduli space of
bundles on curves, \textit{Amer. J. Math.} {\bf 90} (1968) 1200--1208.

\bibitem[Sc1]{Sc1} F. Schaffhauser, Moduli spaces of vector bundles over a 
Klein surface, {\it Geom. Dedicata} {\bf 151} (2011), 187--206.

\bibitem[Sc2]{Sc2} F. Schaffhauser, Real points of coarse moduli schemes of vector 
bundles on a real algebraic curve, {\it Jour. Symp. Geom.} {\bf 10} (2012), 503--534.

\bibitem[Sc3]{Sc3} F. Schaffhauser, On the Narasimhan–Seshadri correspondence for 
real and quaternionic vector bundles, {\it Jour. Diff. Geom.} {\bf 105} (2017), 
119--162.

\bibitem[Tj]{Tj} A. N. Tjurin, Analogues of Torelli's theorem for multidimensional
vector bundles over an arbitrary algebraic curve,
{\it Izv. Akad. Nauk SSSR} {\bf 34} (1970), 338--365.

\end{thebibliography}
\end{document}